\documentclass[12pt]{article}

\usepackage[a4paper,scale=0.75]{geometry}

\usepackage{amsmath}
\usepackage{amssymb}
\usepackage{amsfonts}
\usepackage{color}
\usepackage{biblatex}
\usepackage{amsthm}

\newcommand{\nc}{\newcommand}
\nc{\thref}[1]{Theorem~\ref{theo:#1}}
\nc{\selabel}[1]{\label{sect:#1}}
\nc{\seref}[1]{Section~\ref{sect:#1}}
\nc{\lelabel}[1]{\label{lemm:#1}}
\nc{\leref}[1]{Lemma~\ref{lemm:#1}}
\nc{\prlabel}[1]{\label{prop:#1}}
\nc{\prref}[1]{Proposition~\ref{prop:#1}}
\nc{\colabel}[1]{\label{coro:#1}}
\nc{\coref}[1]{Corollary~\ref{coro:#1}}
\nc{\exlabel}[1]{\label{exam:#1}}
\nc{\exref}[1]{Example~\ref{exam:#1}}
\nc{\delabel}[1]{\label{defi:#1}}
\nc{\deref}[1]{Definition~\ref{defi:#1}}
\nc{\eqlabel}[1]{\label{equa:#1}}
\nc{\relabel}[1]{\label{rema:#1}}
\nc{\reref}[1]{Lemma~\ref{rema:#1}}
\providecommand{\operatorname}[1]{\mathrm{#1}\,}
\nc{\Hom}{\operatorname{Hom}} \nc{\Mor}{\operatorname{Mor}}
\nc{\Aut}{\operatorname{Aut}} \nc{\Ann}{\operatorname{Ann}}
\nc{\Ker}{\operatorname{Ker}} \nc{\Trace}{\operatorname{Trace}}
\nc{\Char}{\operatorname{Char}} \nc{\Mod}{\operatorname{Mod}}
\nc{\End}{\operatorname{End}} \nc{\Spec}{\operatorname{Spec}}
\nc{\Span}{\operatorname{Span}} \nc{\sgn}{\operatorname{sgn}}
\nc{\Id}{\operatorname{Id}} \nc{\Com}{\operatorname{Com}}
\nc{\rank}{\operatorname{rank}} \nc{\Li}{\operatorname{Li}}
\newenvironment{nonitalicproof}{\par\noindent{\bfseries}}{\par}

\let\:=\colon

\newtheorem{de}{Definition}[section]
\newtheorem{pf}{Proof}
\newtheorem{lm}[de]{Lemma}
\newtheorem{cm}{Claim}
\newtheorem{pr}[de]{Proposition}
\newtheorem{co}[de]{Corollary}
\newtheorem{re}[de]{Remark}
\newtheorem{res}[de]{Remarks}
\newtheorem{te}[de]{Theorem}
\newtheorem{ex}[de]{Example}
\newtheorem{exs}[de]{Examples}

\let\:=\colon
\usepackage{graphicx} 

\title{Brennan Conjecture for Basin of Attraction at Infinity}
\author{Yigang Zheng }
\date{}

\begin{document}

\maketitle

\begin{abstract}
    This paper investigates the Brennan Conjecture for  domains $\Omega$ that arise as basins of attraction of a polynomial. We extend the result of Baranski, Volberg, and Zdunik to a broader class of polynomials. We prove that for any monic polynomial of degree $m$ with  sufficiently small 
    non-leading coefficients, Brennan Conjecture holds for its basin of attraction.
\end{abstract}

\section{Introduction}
In \cite{1978The}, Brennan studied the exponent $p\in \mathbb{R}$ for which\\ 
\begin{align}\label{Breenan}
    \int_{\Omega}\left|F'\right|^{p}<+\infty
\end{align}
where $F:\Omega\rightarrow \mathbb{D}$ is a biholomorphic map.
He conjectured that the integral in (\ref{Breenan}) is convergent for all 
$p\in (\frac{4}{3},4)$. For $p=2$ this integral represents the area of the unit disk and is therefore finite. If $\Omega$ is the complement of the slit in the complex plane, then the integral evidently diverges for $p=\frac{4}{3}$ and $p=4$.
It is known that this integral converges when $\frac{4}{3}<p<3.75$ by the work of Hedenmalm and Shimorin in \cite{MR2130416}, improving on earlier result of Bertilsson in \cite{1999On}. There are many other related papers, such as  \cite{Thomas1973On}.

Let $f(z)=\sum_{k=0}^{m}a_{k}z^{k}$ be a polynomial of degree $m\ge 2$ ($a_{m}\neq 0$). We define the basin of attraction of $f$ at infinity to be 
$$\Omega(f)=\{z\in \hat{\mathbb{C}}\mid f^{n}(z)\rightarrow \infty~as~n\rightarrow +\infty\}.$$
If $\Omega=\Omega(f)$ is simply-connected, it is natural to ask whether the Brennan conjecture holds for this domain.

In \cite{1} Baranski, Volberg, and Zdunik proved this to be the case for $m=2$.
\begin{te}\label{BVz}
    Let $f(z)=z^{2}+a_{0}$ be a polynomial of degree $m=2$ such that $\Omega(f)$ is simply-connected, then Brennan conjecture holds for the domain $\Omega(f)$.
\end{te}

\begin{re}
    Every polynomial of degree $m=2$ is conjugate to one of the form $z^{2}+a_{0}$, so Theorem \ref{BVz} covers all degree-2 polynomials. But for $m>2$ this is not the case.                         
\end{re}

In this paper, we first extend Theorem \ref{BVz} to polynomials of the form $f(z)=z^{m}+a_{0}$ for general $m$.

\begin{te}\label{pure generalization of BVZ}
    When $m\ge 3$, for $f(z)=z^{m}+a_{0}$ such that $\Omega(f)$ is simply-connected, Brennan conjecture holds for $\Omega(f)$.
\end{te}

Guided by perturbation ideas, we relax the condition a little bit on the coefficients and obtain a more general result.

\begin{te}\label{more freedom on coefficients}
    For every $m\ge 3$, there exists $\delta=\delta(m)>0$ with the following property:\\
    Let $f(z)=z^{m}+\sum_{k=0}^{m-1}a_{k}z^{k}$ such that $\Omega(f)$ is simply-connected. If $\left|a_{k}\right|<\delta$ for $1\le k\le m-1$, then the Brennan conjecture holds for $\Omega(f)$.
\end{te}

The proofs of these theorems make use of the following result, which is shown by  Lennart Carleson and Theodore in \cite{0Complex}.

\begin{te}\label{lct0}
    Let $f$ be a polynomial of degree at least 2, and $C(f)$ be the set of finite critical points. If $C(f)\subseteq\Omega(f)$, then $\Omega(f)$ is not simply-connected.
\end{te}

A sufficient condition for $\Omega(f)$ to be simply-connected is also given in \cite{0Complex}.

\begin{te}\label{lct1}
    Let $f$ be a polynomial of degree at least 2. If $C(f)\cap\Omega(f)=\emptyset$, then $\Omega(f)$ is simply-connected.
\end{te}

\begin{re}
    If $f=z^{m}+a_{0}$, we have $C(f)=\{0\}$, so $\Omega(f)$ is simply-connected if and only if $0\notin \Omega(f)$. This enables us to obtain more restrictions on the polynomial from simply-connectedness. In section 3 there's a simpler  proof for this special case in Theorem \ref{lct0}, see Theorem \ref{origin criteria}.
\end{re}

As a corollary, we obtain
\begin{co}
     For every $m\ge 3$, there exists $\delta=\delta(m)>0$ with the following property:\\
    Let $f(z)=z^{m}+\sum_{k=0}^{m-2}a_{k}z^{k}$, when $\left|a_{k}\right|<\delta$ for $0\le k\le m-1$, $\Omega(f)$ is simply-connected, and Brennan Conjecture holds for $\Omega(f)$.
\end{co}

We prove this Corollary assuming that Theorem \ref{more freedom on coefficients} holds true.

\begin{proof}
    Choosing arbitrarily $0<\epsilon<\frac{1}{2}$, there exists some $\delta_{1}>0$ such that 
    $$(1-\epsilon)^{m}+\sum_{k=0}^{m-2}\delta_{1}(1-\epsilon)^{k}<1-\epsilon.$$
    Then whenever $\left|a_{k}\right|<\delta_{1}$ we have $\Omega(f)\cap \bar{\mathrm{B}}_{1-\epsilon}(0)=\emptyset$.
    Since the roots change continuously as the polynomial varies, there exists some $\delta_{2}>0$  such that all critical points of $f$ are inside the small disk $\mathrm{B}_{\epsilon}(0)$. Let $\delta=\min\{\delta_{1},\delta_{2}\}$
    Since $\epsilon<1-\epsilon$, we know that $\Omega(f)$ is simply-connected when $\left|a_{k}\right|<\delta$ by Theorem \ref{lct1}.
\end{proof}

The remainder of this paper is organized as follows.  
In section 2, we reduce the convergence of the integral to the convergence of the  series $\sum_{n=1}^{+\infty}m^{-4n}S_{n}$, where $S_{n}=\sum_{k=0}^{m^{n}-1}\left|(f^{n})'(w_{n,k})\right|^2$ and $w_{n,k}$ are in the backward orbit of a given point.
In section 3, we give a simpler proof of the special case of Theorem \ref{lct0} and then prove Theorem \ref{pure generalization of BVZ}, from which one can see that (\ref{Breenan}) holds even for some $p$ slightly larger than 4. This guides us to prove Theorem \ref{more freedom on coefficients} in section 4.

\section{Main reduction}
In this section, we follow the approach of Baranski, Volberg, and Zdunik to reduce the integral to a series, see Theorem \ref{reduction}. We assume $f$ to be a general polynomial of degree $m$, that is, not necessarily of the form $f(z)=z^{m}+a_{0}$.  

\subsection{Koebe's distortion theorem}
One of the theorems that need to be applied is Koebe's distortion theorem, which is shown in \cite{2019Koebe} and is stated as follows:
\begin{te}\label{Koebe}
    Let $h:\mathbb{D}\rightarrow\mathbb{C}$ be a univalent map with $h(0) = 0$ and $h'(0) = 1$. We have the following inequalities:\\
    $$\frac{\left|z\right|}{(1+\left|z\right|)^{2}}
            \le \left|h(z)\right|\le 
             \frac{\left|z\right|}{(1-\left|z\right|)^{2}}$$
             $$\frac{1-\left|z\right|}{1+\left|z\right|}
            \le \left|z\frac{h'(z)}{h(z)}\right|\le 
             \frac{1+\left|z\right|}{1-\left|z\right|}$$
             $$\frac{1-\left|z\right|}{(1+\left|z\right|)^{3}}
            \le \left|h'(z)\right|\le 
             \frac{1+\left|z\right|}{(1-\left|z\right|)^{3}}.$$
\end{te}
With Koebe's distortion theorem, when the unit disk $\mathbb{D}$ is partitioned into several regions, we are able to estimate the derivative at points in one region by its value at a specific  point within that region.

\subsection{Reduction process}
In order to apply Koebe's distortion theorem, we need to invert the biholomorphic map. Let $R=F^{-1}:\mathbb{D}\rightarrow\Omega$ be the inverse of $F$.
By the inverse function theorem and a change of variables, we have 

$$\int_{\Omega}\left|F'\right|^{p}=
\int_{\mathbb{D}}\left|R'\right|^{2-p}.$$

After its composition with the Mobius map, we may assume $R(0)=\infty$, then $R$ has a first-order pole at $0$.
In our case, the domain $\Omega=\Omega(f)$, let $g=R^{-1}\circ f \circ R:\mathbb{D}\rightarrow \mathbb{D}$ be the conjugation of $f$ with $R$. We have the following lemma.

\begin{lm}\label{central}
If $f$ is a polynomial of degree $m\ge 2$, then $g(z)=az^{m}$ for some constant $a$ with $\left|a\right|=1$.
\end{lm}

\begin{proof}
Since $R$ has a first-order pole at $0$, $\lim_{z\rightarrow 0}zR(z)=b_{0}$ for some non-zero constant $b_{0}\in\mathbb{C}$.  
Since $f$ is a polynomial of degree $m$, $\lim_{z\rightarrow +\infty}\frac{f(z)}{z^{m}}=b_{1}$ is the leading coefficient of $f$.  
Let $h(z)=\frac{g(z)}{z^{m}}$.  Then $h$ has a removable singularity at $0$ and has value $\frac{1}{b_{1}b_{0}^{m-1}}$ there, so $h$ is nowhere zero on $\mathbb{D}$.  
Since $f$ and $R$ can be extended to the boundary $\partial\Omega$ and $\partial\mathbb{D}$ respectively, so can $h$ to $\partial\mathbb{D}$.  
For any $z\in\partial\mathbb{D}$, we have $\left|h(z)\right|=1$ by definition, so applying maximal principle to $h$ and $h^{-1}$, we know that $h$ is a constant and the claim follows.
\end{proof}

And what's more, after the conjugation by a rotation, we may assume $a=1$ in Lemma \ref{central}.

It is natural to divide the unit disk into fundamental domains of the action $g:\mathbb{D}\rightarrow\mathbb{D}$. To be precise, we partition the unit disk as follows.

Fix a real number $0<r_{0}<1$, let 
$$P_{0}=\{z\in \mathbb{C}~|~r_{0}\le \left|z\right|\le r_{0}^{\frac{1}{m}} \} $$

and 
$$P_{n}=g^{-n}(P_{0})=\{z\in \mathbb{C}~|~
r_{0}^{\frac{1}{m^{n}}}\le \left|z\right|\le r_{0}^{\frac{1}{m^{n+1}}} \}.$$
Fix a point $z_{0}$ with $\left|z_{0}\right|=r_{0}$. For each $n\ge 0$ there are $m^{n}$ preimages of $z_{0}$ under $g^{n}$, which are equally distributed, so we further divide  $P_{n}$ into $m^{n}$ parts as follows:
$$P_{n,k}=\left\{re^{i\theta}|r_{0}^{\frac{1}{m^{n}}}\le r\le r_{0}^{\frac{1}{m^{n+1}}},k\frac{2\pi}{m^{n}}\le\theta<(k+1)\frac{2\pi}{m^{n}}\right\}$$
for $0\le k<m^{n}$, and let $z_{n,k}\in P_{n,k}$ be preimages of $z_{0}$ under $g^{n}$.

The fact that $$\lim_{n\rightarrow+\infty}\frac{1-r_{0}^{\frac{1}{m^{n}}}}{\frac{1}{m^{n}}}=-\log(r_{0})$$ tells us that there exist $M<+\infty$ and $\gamma>0$ such that we can cover each region $P_{n,k}$ by $M$ disks $\mathrm{B}_{(1-\gamma)t}(x)$ with $\mathrm{B}_{t}(x)\subseteq\mathbb{D}\setminus\{0\}$, so we obtain 

\begin{align}\label{one point estimate}
    (2\gamma^{-3})^{-M}\left|R'(z_{n,k})\right|
\le\left|R'(z)\right|
\le(2\gamma^{-3})^{M}\left|R'(z_{n,k})\right|
\end{align}
 using Koebe's distortion theorem.

Denote by $w_{0}=R(z_{0})$ and $w_{n,k}=R(z_{n,k})$, which means that $w_{n,k}$ ($0\le k<m^{n}$) are preimages of $w_{0}$ under $f^{n}$.\\

Applying the the chain rule to the diagram $R\circ g^{n}=f^{n}\circ R$ on $P_{n}$, we have the approximation 
$$C_{1}\frac{m^{n}}{\left|(f^{n})'(w_{n,k})\right|}
\le\left|R'(z_{n,k})\right|\le 
C_{2}\frac{m^{n}}{\left|(f^{n})'(w_{n,k})\right|}$$
for some constants $0<C_{1}<C_{2}<+\infty$ depending only on $m$ and $r_{0}$.

Observe that (\ref{Breenan}) for $p=4$ implies that for $2<p<4$, as $u^{p}<u^{2}+u^{4}$ for $u>0$, so we can assume $p=4$ and then the integral equals $\int_{\mathbb{D}}\left|R'\right|^{-2}$.\\

Denote by $S_{n}=\sum_{k=0}^{m^{n}-1}\left|(f^{n})'(w_{n,k})\right|^2$, we reach the following result using (\ref{one point estimate}):

\begin{te}\label{reduction}
    There exists constants $C_{1},C_{2}<+\infty$ depending only on $m$ and $r_{0}$ such that for any $2\le p\le 4$, $\int_{\Omega}\left|F'\right|^{p}<C_{1}\sum_{n=1}^{+\infty}m^{-4n}S_{n}+C_{2}$.
\end{te}

\begin{re}
    According to the above approximation, it suffices to show that $$\limsup_{n\rightarrow+\infty}\frac{S_{n+1}}{S_{n}}<m^{4},$$ as the chain rule helps us derive the relation between $S_{n}$ and $S_{n+1}$.
\end{re}

Since $(f^{n+1})'(z)=(f^{n})'(f(z))f'(z)$, we have
\begin{align}
    S_{n+1}=&\notag\sum_{k=0}^{m^{n+1}-1}\left|(f^{n+1})'(w_{n+1,k})\right|^{2}\\
    =&\notag\sum_{k=0}^{m^{n+1}-1}\left|(f^{n})'(f(w_{n+1,k}))f'(w_{n+1,k})\right|^{2}\\
    =&\notag\sum_{k=0}^{m^{n+1}-1}\left|(f^{n})'(f(w_{n+1,k}))\right|^{2}
    \left|f'(w_{n+1,k})\right|^{2}\\
    =&\notag\sum_{k=0}^{m^{n}-1}\left|(f^{n})'(f(w_{n,k}))\right|^{2}
    \sum_{w\in\mathbb{C},f(w)=w_{n,k}}
    \left|f'(w)\right|^{2}\\
    \le&\notag\sum_{k=0}^{m^{n}-1}\left|(f^{n})'(f(w_{n,k}))\right|^{2}
    \max\{
    \sum_{w\in\mathbb{C},f(w)=w_{n,k}}
    \left|f'(w)\right|^{2}
    |0\le k<m^{n}\}\\
    =&\notag S_{n}\max\{
    \sum_{w\in\mathbb{C},f(w)=w_{n,k}}
    \left|f'(w)\right|^{2}\}.
\end{align}

\begin{lm}\label{closetoboundary}
    For any compact subset $K\subseteq \Omega(f)$, we can find $N\in \mathbb{N}$ such that for any $n\ge N$, $w_{n,k}\notin K$ for every $k$.
\end{lm}

\begin{proof}
Choose a neighborhood $U$ of $\infty$ such that
$\infty\in U\subseteq \Omega(f)$, $w_{0}\notin U$, and 
$f^{-1}(U)\supseteq U$.
From the definition of $\Omega(f)$, we know that $\cup_{n=0}^{+\infty} f^{-n}(U)=\Omega(f)\supseteq K$.
By the compactness of $K$, we can find $N\in \mathbb{N}$ such that $\cup_{n=0}^{N} f^{-n}(U)\supseteq K$, that is, $f^{-N}(U)\supseteq K$ since $f^{-n}(U)\subseteq f^{-(n+1)}(U)$.
So $w_{n,k}\notin K$ for every $k$ when $n\ge N$.  
\end{proof}

\begin{re}
    Lemma \ref{closetoboundary} implies that for large $n$ the points $w_{n,k}$ lie arbitrarily close to $\partial\Omega$ .
\end{re}

\section{The calculation for specific polynomials}
In this section, we prove Theorem \ref{pure generalization of BVZ}. From the expression of $f$ we have $f(z)=f(\exp(\frac{2\pi ki}{m})z)$ for $k\in\mathbb{Z}$, so 
$\Omega(f)=\exp(\frac{2\pi ki}{m})\Omega(f)$, which means that $\Omega(f)$ is rotationally symmetric about the origin.
Since $t^{m}-t$ is an increasing function for $t\in [1,+\infty)$, we have $\Omega(f)\supseteq 
\{z\in\mathbb{C}~|~\left|z\right|>\alpha(\left|a_{0}\right|)\}$, where $\alpha(a)$ is the root of $t^{m}-t-a=0$ in $[1,+\infty)$ for $a\ge 0$.

As is pointed out by in \cite{1} Baranski, Volberg, and Zdunik, the next lemma follows from Theorem \ref{lct0}, but here we give a simpler and more elementary proof.

\begin{lm}\label{origin criteria}
    If $\Omega(f)$ is simply connected,  then $0\notin \Omega(f)$.
\end{lm}
\begin{proof}
First note that since $\deg(f)=m\ge 2$, we can find $w\in \mathbb{C}$ such that $f(w)=w$ and since $w\notin \Omega(f)$, its complement $V=\hat{\mathbb{C}}\backslash\Omega(f)$ is non-empty.

Assume the contrary that $0\in \Omega(f)$, then $0\notin V$.  
Choose $0<r_{1}<+\infty$ such that $V\subseteq \mathrm{B}_{r_{1}}(0)$.

Since $\Omega(f)$ is simply connected, we can find a path 
$\gamma :[0, 1]\rightarrow \Omega(f)$ with 
$\gamma(0)=0$,  $\gamma(1)=r_{1}$. Consider the set

$$E=\{z\in \mathbb{C}\mid \left|z\right|=r_{1}\}\cup 
\Big(\bigcup_{l=0}^{m-1}\big(\exp(l\frac{2\pi i}{m})\gamma([0, 1])\big)\Big).$$

$E$ divides $V$ into at least $m\ge 2$ parts,  which means that $V$ is not connected,  contradicting the simply connectedness of $\Omega(f)$.
So we must have $0 \notin \Omega(f)$.
\end{proof}

By Theorem \ref{reduction}, to prove Theorem \ref{pure generalization of BVZ} it suffices to show that $\{S_{n}\}_{n}$ form a geometric series, to be precise, $\limsup_{n\rightarrow+\infty}\frac{S_{n+1}}{S_{n}}<m^{4}$.

From $0 \notin \Omega(f)$ we get $a_{0}=f(0) \notin \Omega(f)$, so $\left|a_{0}\right|\le \alpha(\left|a_{0}\right|)$. By monotonicity of $t^{m}-t$, we have $\left|a_{0}\right|^{m}-\left|a_{0}\right|\le \left|a_{0}\right|$, which gives 
$\left|a_{0}\right|\le 2^{\frac{1}{m-1}}$.

Because $f(z)=z^{m}+a_{0}$, it follows that the derivative $f'(z)=mz^{m-1}$ has a nice form. This also implies the following relations: 
$$f(u)=v\Longrightarrow \left|u\right|^{m}=\left|v-a_{0}\right|
\Longrightarrow 
\left|f'(u)\right|=m\left|u\right|^{m-1}=m\left|v-a_{0}\right|^{\frac{m-1}{m}}.$$   
So we can proceed as follows.

For any $\epsilon>0$, by Lemma \ref{closetoboundary}, there is $N\in \mathbb{Z}_{\ge 1}$ such that for any $n\ge N$, $0\le k<m^{n}-1$, we have $\left|w_{n,k}\right|<\alpha(\left|a_{0}\right|)+\epsilon\le 2^{\frac{1}{m-1}}+\epsilon$.For any $n\ge N$, we have 
\begin{align}
    S_{n+1}&\notag=\sum_{k=0}^{m^{n}-1}\left|(f^{n})'(w_{n,k})\right|^{2}
    \sum_{w\in\mathbb{C},f(w)=w_{n,k}}
    \left|f'(w)\right|^{2}\\
    &\notag=m^{3}\sum_{k=0}^{m^{n}-1}
    \left|(f^{n})'(w_{n,k})\right|^{2}\left|w_{n,k}-a_{0}\right|^{\frac{2(m-1)}{m}}.
\end{align}

So we obtain  the following estimate

\begin{align}
    S_{n+1}\le m^{3}\sum_{k=0}^{m^{n}-1}
    \left|(f^{n})'(w_{n,k})\right|^{2}(2^{\frac{m}{m-1}}+\epsilon)^{\frac{2(m-1)}{m}}\le (4+c_{1}\epsilon+c_{2}\epsilon^{2})m^{3}S_{n} 
\end{align}
for some constants $c_{1},c_{2}$, which are independent of $m$(though we do not use this fact).\\

When $m\ge 5$, we have 
$$\limsup_{n\rightarrow+\infty}\frac{S_{n+1}}{S_{n}}\le 
4m^{3}<m^{4}$$
letting $\epsilon\rightarrow 0$, and Theorem \ref{pure generalization of BVZ} is proved. 

This estimate is not precise enough when $m=3,4$. For this we develop a sharper estimate, which is sufficient for all $m\ge 3$

\begin{align}
    S_{n+1}&\notag=m^{2}\sum_{k=0}^{m^{n}-1}\sum_{l=0}^{m-1}
    \left|(f^{n})'(\exp(l\frac{2\pi i}{m})w_{n,k})\right|^{2}\left|\exp(l\frac{2\pi i}{m})w_{n,k}-c\right|^{\frac{2(m-1)}{m}}\\
    &\notag=m^{2}\sum_{k=0}^{m^{n}-1}\sum_{l=0}^{m-1}
    \left|(f^{n})'(w_{n,k})\right|^{2}\left|\exp(l\frac{2\pi i}{m})w_{n,k}-c\right|^{\frac{2(m-1)}{m}}\\
    &\notag\le m^{3}\sum_{k=0}^{m^{n}-1}\left|(f^{n})'(w_{n,k})\right|^{2}
    (\frac{1}{m}\sum_{l=0}^{m-1}\left|\exp(l\frac{2\pi i}{m})w_{n,k}-c\right|^{2})^{\frac{m-1}{m}}\\
    &\notag=m^{3}\sum_{k=0}^{m^{n}-1}\left|(f^{n})'(w_{n,k})\right|^{2}
    (\left|w_{n,k}\right|^{2}
    +\left|c\right|^{2})
    ^{\frac{m-1}{m}}\\
    &\notag\le m^{3}\sum_{k=0}^{m^{n}-1}\left|(f^{n})'(w_{n,k})\right|^{2}
    ((2^{\frac{1}{m-1}}+\epsilon)^{2}
    +2^{\frac{2}{m-1}})
    ^{\frac{m-1}{m}}\\
    &\notag\le m^{3}((2^{\frac{1}{m-1}}+\epsilon)^{2}
    +2^{\frac{2}{m-1}})
    ^{\frac{m-1}{m}}S_{n}.
\end{align}

Leaving $\epsilon\rightarrow 0^{+}$, we reach the bound
$$\limsup_{n\rightarrow+\infty}\frac{S_{n+1}}{S_{n}}\le 
m^{3}2^{\frac{m+1}{m}}.$$

We have $2^{\frac{m+1}{m}}\le 2^{\frac{4}{3}}<3\le m$, so the limit we just obtained is strictly less than $m^{4}$.\\
So we prove Theorem \ref{pure generalization of BVZ} according to Theorem \ref{reduction}.\\

\section{Small perturbation of $z^{m}+a_{0}$}
In the previous section, we've dealt with the polynomial of the form $f(z)=z^{m}+a_{0}$ and proved Theorem \ref{pure generalization of BVZ}. Now we prove Theorem \ref{more freedom on coefficients}.\\

Let $f=\sum_{k=0}^{m}a_{k}z^{k}$ be a polynomial of degree $m$ ($a_{m}\neq 0$), by the conjugation using linear maps, we may assume that $a_{m}=1$. For $\delta>0$ small enough, we want to obtain similar properties as shown in section 3 under the condition described in the following definition:

\begin{de}
    We say $f$ satisfies the $\delta$-condition if $\left|a_{k}\right|<\delta$ for $1\le k\le m-1$.
\end{de}

\begin{re}
    Note that there's no constraint on the constant term $a_{0}$.
\end{re}

We show that a similar argument works when $f$ satisfies the  $\delta$-condition for small $\delta$. 

\begin{lm}\label{limited disturb of critical points}
    For any $\epsilon>0$, there exists $\delta>0$ such that if $f$ satisfies the $\delta$-condition, all finite critical points of $f$ are inside $\mathrm{B}_{\epsilon}(0)$.
\end{lm}

\begin{proof}
    We can express the derivative $f'$ as $f'(z)=\sum_{k=1}^{m}ka_{k}z^{k-1}$. For every 
    $z\in \mathbb{C}$, we have $$\left|f'(z)\right|\ge m\left|z\right|^{m-1}-\sum_{k=1}^{m-1}k\left|a_{k}\right|\left|z\right|^{k-1}\ge m\left|z\right|^{m-1}(1-\sum_{l=1}^{m-1}\frac{\left|a_{m-l}\right|}{\left|z\right|^{l}}).$$
    
    When $\left|z\right|\ge\epsilon$, 
    $$\left|f'(z)\right|\ge  m\left|z\right|^{m-1}(1-\sum_{l=1}^{m-1}\frac{\left|a_{m-l}\right|}{\epsilon^{l}})$$
     for $\left|a_{k}\right|$($1\le k\le m-1$) small enough we have 
    $$1-\sum_{l=1}^{m-1}\frac{\left|a_{m-l}\right|}{\epsilon^{l}}>0$$ 
    and thus $f'(z)\neq 0$, that is, $z$ is not the critical point of $f$.
\end{proof}

By Lemma \ref{limited disturb of critical points} and Theorem \ref{lct0}, we know that for any $\epsilon>0$, there exists $\delta>0$ small enough such that $\Omega(f)\nsupseteq \mathrm{B}_{\epsilon}(0)$ if $f$ satisfies the $\delta$-condition and $\Omega(f)$ is simply-connected, which is similar to Lemma \ref{origin criteria}. \\

We denote by 
$$b(f)=\sup\{\left|z\right|~|~z\in\mathbb{C}\setminus \Omega(f)\}$$ and we have the following lemma
\begin{lm}\label{boundary gap}
    For any $\epsilon>0$, there exists $\delta>0$ such that  $$\inf\{\left|a_{0}\right|-b(f)\}\ge\frac{\epsilon}{3}$$ where the infimum is taken over those polynomials $f$ satisfying the $\delta$-condition and $\left|a_{0}\right|\ge 2^{\frac{1}{m-1}}+\epsilon$.
\end{lm}

\begin{re}
    As is shown in section 3, there's always a gap between 
    $\partial\Omega(f)$ and $\partial \mathrm{B} _{\left|a_{0}\right|}(0)$ when $\left|a_{0}\right|>2^{\frac{1}{m-1}}$, and Lemma \ref{boundary gap} tells us that similar situation arises.
\end{re}

\begin{proof}
    Let $h(t)=t^{m}-\sum_{k=0}^{m-1}\left|a_{k}\right|t^{k}-t$, we have $h'(t)=mt^{m-1}-\sum_{k=1}^{m-1}k\left|a_{k}\right|t^{k}-1$.
    
    For $t\ge 1$, we have 
    $$\frac{h'(t)}{t^{m-1}}=m-\sum_{1=0}^{m-1}k\left|a_{k}\right|t^{k-m}-t^{1-m}\ge m-1-\sum_{1=0}^{m-1}k\left|a_{k}\right|.$$
    So there exists some $\delta_{1}>0$ such that when the $\delta_{1}$-condition is satisfied, $\frac{h'(t)}{t^{m-1}}>0$ for $t\ge 1$, which is $h'(t)>0$.
    
    For any $\epsilon>0$, since $(2^{\frac{1}{m-1}}+\frac{2}{3}\epsilon)^{m}>2(2^{\frac{1}{m-1}}+\epsilon)$, we can find $0<\delta_{2}<\delta_{1}$ such that 
    $$(2^{\frac{1}{m-1}}+\frac{2}{3}\epsilon)^{m}-\sum_{k=1}^{m-1}\delta(2^{\frac{1}{m-1}}+\frac{2}{3}\epsilon)^{k}-(2^{\frac{1}{m-1}}+\epsilon)>\frac{\epsilon}{3}.$$
    Then whenever the $\delta_{2}$-condition is satisfied and $\left|a_{0}\right|\ge 2^{\frac{1}{m-1}}+\epsilon$, we have $h(t)>\frac{\epsilon}{3}$ for 
    $t\ge \left|a_{0}\right|-\frac{\epsilon}{3}$.
    
    According to the above analysis, one can show inductively that 
    $\left|f^{n}(z)\right|\ge\left|z\right|+\frac{n}{3}\epsilon_{2}$ for $\left|z\right|\ge \left|a_{0}\right|-\frac{\epsilon}{3}$, which is the lemma.
\end{proof}

\begin{lm}\label{limited effect on the constant term}
    For any $\epsilon>0$, there exists $\delta>0$ such that as long as the $\delta$-condition is satisfied and $\Omega(f)$ is simply connected, we have $\left|a_{0}\right|<2^{\frac{1}{m-1}}+\epsilon$.\\
\end{lm}

\begin{proof}
    We assume the contrary that 
    $\left|a_{0}\right|\ge2^{\frac{1}{m-1}}+\epsilon$.
    We choose $\delta_{1}>0$ corresponding to $\epsilon$ in Lemma \ref{boundary gap}, and choose $t>0$ small enough such that 
    $t^{m}+\sum_{k=1}^{m-1}\left|a_{k}\right|t^{k}<\frac{\epsilon}{3}$.
    
    Choose $0<\delta_{2}<\delta_{1}$ corresponding to $\epsilon=t$ in Lemma \ref{limited disturb of critical points}, then every critical point is inside $\mathrm{B}_{t}(0)$
    if $f$ satisfies the $\delta_{2}$-condition.
    
    On the other hand, for any $z\in \mathrm{B}_{t}(0)$, 
    $$\left|f(z)-a_{0}\right|\le t^{m}+\sum_{k=1}^{m-1}\left|a_{k}\right|t^{k}<\frac{\epsilon}{3}$$
    which gives $\left|f(z)\right|>\left|a_{0}\right|-\frac{\epsilon}{3}$ by the triangle inequality.
    
    From the way $\delta_{1}>0$ is chosen we know that 
    $$\Omega(f)\supseteq\{z\in\mathbb{C}~|~\left|z\right|\ge \left|a_{0}\right|-\frac{\epsilon}{3}\}$$ if $f$ satisfies the $\delta_{2}$-condition, which means $f(z)\in \Omega(f)$ and immediately implies that $z\in \Omega(f)$. Since $z\in \mathrm{B}_{t}(0)$ is chosen arbitrarily, this contradicts the simply-connectedness of $\Omega(f)$.
\end{proof}

\begin{lm}\label{slight change on the domain}
    For any $\epsilon>0$, there exists $\delta>0$ such that if $f$ satisfies the $\delta$-condition and $\Omega(f)$ is simply connected, we have $\Omega(f)\supseteq\{z\in\mathbb{C}~|~\left|z\right|\ge 2^{\frac{1}{m-1}}+\epsilon\}$.
\end{lm}
\begin{proof}
    Choose $\delta_{1}>0$ corresponding to $\epsilon$ in Lemma \ref{limited effect on the constant term}. Then whenever $f$ satisfies the $\delta_{1}$-condition and $\Omega(f)$ is simply-connected, we have $\left|a_{0}\right|<2^{\frac{1}{m-1}}+\epsilon$.

    Let $h(t)=t^{m}-\sum_{k=0}^{m-1}\left|a_{k}\right|t^{k}-t$, as mentioned in the proof of Lemma \ref{boundary gap}, there exists $0<\delta_{2}<\delta_{1}$ such that for any $t\ge 2^{\frac{1}{m-1}}+\epsilon\}$, $h(t)\ge h(2^{\frac{1}{m-1}}+\epsilon)>0$. Then one can show inductively that $\left|f^{n}(z)\right|\ge \left|z\right|+nh(2^{\frac{1}{m-1}}+\epsilon)$for $\left|z\right|\ge 2^{\frac{1}{m-1}}+\epsilon$, which is the Lemma.    
\end{proof}

\subsection{The proof of Theorem \ref{more freedom on coefficients} when $m\ge 5$}
Similar as in section 3, the estimate is simpler when $m\ge 5$.
Choose $\epsilon>0$ such that $2^{\frac{1}{m-1}}+\epsilon<5^{\frac{1}{2(m-1)}}\le m^{\frac{1}{2(m-1)}}$. 
Choose $\delta_{1}>0$ fits $\epsilon$ both in Lemma \ref{limited effect on the constant term} and in Lemma \ref{slight change on the domain}, and choose $0<\delta_{2}<\delta_{1}$ such that 
$$(2^{\frac{1}{m-1}}+\epsilon)^{m-1}
+\delta_{2}\sum_{k=1}^{m-2}(2^{\frac{1}{m-1}}+\epsilon)^{k-1}<\sqrt{5}.$$
If $f$ satisfies the $\delta_{2}$-condition, $$\limsup_{n\rightarrow+\infty}
\max\{\sum_{w\in\mathbb{C},f(w)=w_{n,k}}
    \left|f'(w)\right|^{2}
    ~|~0\le k<m^{n}\}<m(\sqrt{5}m)^{2}=5m^{3}\le m^{4}$$
and thus $\limsup_{n\rightarrow+\infty}\frac{S_{n+1}}{S_{n}}< m^{4}$, which leads to Theorem \ref{more freedom on coefficients} by Theorem \ref{reduction}.\\

\subsection{The proof including $m=3,4$}
In the last subsection the case where $m=3,4$ isn't taken into account, for this case, we need more complicated procedures to get a sharper estimate, which also works for $m\ge 5$.

Let $g(z)=z^{m}+a_{0}$, we compare the preimages of a certain point under $f$ and $g$ and their iterations. Recall that for the domain $\Omega(g)$ we obtained 
$$\sum_{g(g(w'))=w}\left|g'(g(w'))\right|^{2}\left|g'(w')\right|^{2}
\le C\sum_{g(w')=w}\left|g'(w')\right|^{2}$$
when $\left|a_{0}\right|\le 2^{\frac{1}{m-1}}$ and $dist(w,\partial\Omega(g))\le\epsilon$ for some $C<m^{4}$ and $\epsilon>0$, so we consider the following function
$$u_{\delta}:\mathbb{C}\times\mathbb{C}\setminus D\rightarrow
[0,+\infty]$$
defined as 
$$u_{\delta}(w,a)=\sup\bigg\{\left|
\frac{\sum_{f(f(w'))=w}\left|f'(f(w'))\right|^{2}\left|f'(w')\right|^{2}}{\sum_{f(w')=w}\left|f'(w')\right|^{2}}
-\frac{\sum_{g(g(w'))=w}\left|g'(g(w'))\right|^{2}\left|g'(w')\right|^{2}}{\sum_{g(w')=w}\left|g'(w')\right|^{2}}
\right|\bigg\}$$
where the supreme is taken over polynomials $f$ satisfying  the $\delta$-condition, 
and $a_{0}=a$. $D=\{(w,a)\in \mathbb{C}\times\mathbb{C}~|~w=a\}$ is the diagonal set. The summation above is counted in multiplicity as roots of corresponding polynomials.\\

By the continuity of the roots with respect to the polynomial, we know that $u_{\delta}$ is continuous with respect to both variables and converges pointwisely to $0$ as $\delta\rightarrow 0^{+}$, so $\{u_{\delta}|\delta>0\}$ converges uniformly on any compact subset of $\mathbb{C}\times\mathbb{C}\setminus D$.\\

Choose $\epsilon>0$ such that $2(2^{\frac{1}{m-1}}+2\epsilon)^{2}<8$, and choose $\delta_{1}>0$ fits $\epsilon$ in Lemma \ref{boundary gap}, Lemma \ref{limited effect on the constant term}  and Lemma \ref{slight change on the domain}. Then whenever $\Omega(f)$ is simply-connected and $f$ satisfies the $\delta_{1}$-condition , we have $\left|a_{0}\right|\le 2^{\frac{1}{m-1}}+\epsilon<2$.\\

For any $w\in\mathbb{C}$ with $g(g(w))=a_{0}$, we have 
$$g(w)^{m}+a_{0}=a_{0}\Longrightarrow g(w)=0\Longrightarrow 
\left|w\right|=\left|a_{0}\right|^{\frac{1}{m}}\Longrightarrow 
\left|g'(w)\right|=m\left|a_{0}\right|^{\frac{m-1}{m}}.$$

So we consider the following function for $\epsilon'>0$
$$v_{\epsilon'}:\mathbb{C}\rightarrow[0,+\infty]$$
defined by $$v_{\epsilon'}(a)=
\sup\bigg\{\left|\left|f'(w)\right|-m\left|a\right|^{\frac{m-1}{m}}\right|\bigg\}$$
where the supreme is taken over all polynomials $f$  satisfying the $\epsilon'$-condition and $a_{0}=a$, and the points $w\in\mathbb{C}$ with $\left|f(f(w))-a\right|<\epsilon'$.\\

By the continuity of the roots with respect to the polynomial, $v_{\epsilon'}$ is continuous and converges pointwisely to $0$ as $\epsilon'\rightarrow 0^{+}$. So it converges uniformly on $\bar{\mathrm{B}}_{2}(0)$.\\

Choose $\epsilon'>0$ such that $v_{\epsilon'}<\epsilon$ on $\bar{\mathrm{B}}_{2}(0)$. 
Let $K=\{(w,a)~|~\left|w_{*}\right|,\left|a\right|\le 2,
\left|w-a\right|\ge\epsilon'\}$. Then $\{u_{\delta}|\delta>0\}$
converges uniformly on $K$ as 
$K\subseteq\mathbb{C}\times\mathbb{C}\setminus\ D$ is a compact subset.\\

Choose $0<\delta_{2}<\delta_{1}$ such that $u_{\delta_{2}}<\epsilon$ on $K$. Then whenever $f$ satisfies  the $\delta_{2}$-condition , we have the following for a given point $w\in\mathbb{C}$

\begin{align}
    &\frac{\sum_{f(f(w'))=w}\left|f'(f(w'))\right|^{2}\left|f'(w')\right|^{2}}{\sum_{f(w)=w_{*}}\left|f'(w')\right|^{2}}\\ \notag &\le
\frac{\sum_{g(g(w'))=w}\left|g'(g(w'))\right|^{2}\left|g'(w')\right|^{2}}{\sum_{g(w')=w}\left|g'(w')\right|^{2}}+\epsilon\\ \notag &\le
 m^{3}(2(2^{\frac{1}{m-1}}+\epsilon)^{2})^{\frac{m-1}{m}}+\epsilon<m^{4}
\end{align}
when $w$ is close enough to $\partial\Omega(f)$ and $\left|w-a_{0}\right|\ge\epsilon'$.

When $\left|w-a_{0}\right|<\epsilon'$, by the definition of  $\epsilon'$ we have $\max\bigg\{\left|f'(w)-m\left|a_{0}\right|^{\frac{m-1}{m}}\right|~|~w\in\mathbb{C}, f(f(w))=w_{*}\bigg\}<\epsilon$ if $f$ satisfies the $\epsilon'$-condition.\\

In this case, 
\begin{align}
    &\frac{\sum_{f(f(w'))=w}\left|f'(f(w'))\right|^{2}\left|f'(w')\right|^{2}}{\sum_{f(w')=w}\left|f'(w')\right|^{2}}\\ \notag &\le
    \max\{\sum_{f(w'')=w'}\left|f'(w'')\right|^{2}~|~f(w')=w\}\\
    \notag &\le m(m\left|a_{0}\right|^{\frac{m-1}{m}}+\epsilon)^{2}<m^{4}.
\end{align}

Let $\delta=\min\{\delta_{2},\epsilon'\}$,
we have $\limsup_{n\rightarrow+\infty}\frac{S_{n+1}}{S_{n}}<m^{4}$ if $f$ satisfies the $\delta$-condition. And we prove Theorem \ref{more freedom on coefficients} according to Theorem \ref{reduction}.

\section{Acknowledgement}
I would like to show my gratitude to everyone who contribute to my thesis. First, I am deeply grateful to my thesis advisor, Vladimir Markovic and professor Yitwah Cheung for inspiring me with valuable ideas. I also wish to thank professor Haakan Hedenmalm for the interesting discussions we shared. My thanks also go to my fellow classmate, Tingyu Zhang, for helping me become familiar with Latex. I appreciate the support of the school library at Tsinghua University for providing access to academic materials. Additionally, I am thankful to Qiuzhen College for equipping me with the knowledge and skills necessary to form my thesis.

\printbibliography

\end{document}